\documentclass[a4paper,11pt]{article}

\usepackage{amsfonts,amsmath,amssymb,amsthm,epsfig,psfrag}
\usepackage[all]{xy}

\usepackage{psfrag, epsfig, amsmath,amssymb,amsthm,accents}

\newtheorem{proposition}{Proposition}
\newtheorem{theorem}{Theorem}

\newtheorem{lemma}{Lemma}
\newtheorem{corollary}{Corollary}

\newtheorem{definition}{Definition}

\def\N{{\mathbb N}}
\def\EE{{\mathbb E}}
\def\PP{{\mathbb P}}
\def\RR{{\mathbb R}}
\def\ind{{\mathbf{1}}}
\def\ore{\overrightarrow{e}}
\def\orE{\overrightarrow{E}}

\def\d{\partial}
\def\bx{\bold{x}}
\def\by{\bold{y}}

\def\bh{\bold{h}}

\def\bB{\bold{B}}

\def\Rcal{{\mathcal R}}
\def\Dcal{{\mathcal D}}

\def\Scal{{\mathcal S}}

\def\cG{{\mathcal G}}
\def\cGs{{\mathcal G}_{\star}}

\def\bac{\backslash}
\def\lwc{\leadsto}

\def\EE{{\mathbb E}}

\def\PP{{\mathbb P}}

\def\per{{\mathrm{per}}}

\def\ind{{\mathbf{1}}}

\def\R{{\mathbb R}}

\def\N{{\mathbb N}}

\newcommand*\lbf{\ensuremath{\boldsymbol\ell}}
\def\rbf{{\mathbf r}}

\def\cG{{\mathcal G}}

\def\cS{\Sigma}

\def\cU{{\mathcal U}}

\def\cP{{\mathcal P}}

\def\cL{{\mathcal L}}

\newcommand{\BEAS}{\begin{eqnarray*}}
\newcommand{\EEAS}{\end{eqnarray*}}
\newcommand{\BEA}{\begin{eqnarray}}
\newcommand{\EEA}{\end{eqnarray}}
\newcommand{\BEQ}{\begin{equation}}
\newcommand{\EEQ}{\end{equatioyn}}
\newcommand{\BIT}{\begin{itemize}}
\newcommand{\EIT}{\end{itemize}}
\newcommand{\BNUM}{\begin{enumerate}}
\newcommand{\ENUM}{\end{enumerate}}

\begin{document}
%
\title{Counting matchings in irregular bipartite graphs and random lifts}

\author{M. Lelarge\footnote{INRIA-ENS, Paris, France,
email: marc.lelarge@ens.fr}}
\date{}

\maketitle

\begin{abstract}
We give a sharp lower bound on the number of matchings of a given size
in a bipartite graph. When specialized to regular bipartite graphs,
our results imply Friedland's Lower Matching Conjecture and
Schrijver's theorem proven by Gurvits and Csikv\'ari.
Indeed, our work extends the recent work of Csikv\'ari done for
regular and bi-regular bipartite graphs. Moreover, our lower bounds
are order optimal as they are attained for a sequence of
$2$-lifts of the original graph as well as for random $n$-lifts of the
original graph when $n$ tends to infinity.

We then extend our results to permanents and subpermanents sums. For
permanents, we are able to recover the
lower bound of Schrijver recently proved by Gurvits using stable
polynomials. 
Our proof is algorithmic and borrows ideas from the theory of local
weak convergence of graphs, statistical physics and covers of
graphs. We provide new lower bounds for subpermanents sums and obtain
new results on the number of matching in random $n$-lifts with some
implications for the matching measure and the spectral measure of
random $n$-lifts as well as for the spectral measure of infinite trees.
\end{abstract}

\section{Introduction}
Recall that a $n\times n$ matrix $A$ is called doubly stochastic if it
is nonnegative entrywise and each of its columns and rows sums to
one. Also the permanent of a $n\times n$ matrix $A$ is defined as
\BEAS
\per(A) =\sum_{\sigma\in \Scal_n}\prod_{i=1}^n a_{i, \sigma(i)},
\EEAS
where the summation extends over all permutation $\sigma$ of
$\{1,\dots, n\}$.
The main result proved in \cite{sch98} is the following theorem:
\begin{theorem}\label{th:sch}(Schrijver \cite{sch98})
For any doubly stochastic $n\times n$ matrix $A=(a_{i,j})$, we define
$\tilde{A}=(\tilde{a}_{i,j}=a_{i,j}(1-a_{i,j}))$ and we have
\BEA
\label{eq:sch1}\per(\tilde{A})\geq \prod_{i,j}(1-a_{i,j}).
\EEA
\end{theorem}

It is proved in \cite{gur11,6978993} that this theorem implies:
\begin{theorem}\label{th:gurv}
Let $A$ be a non-negative $n\times n$ matrix. Then, we
have
\BEA
\label{eq:lnper}\ln \per(A) \geq \max_{x\in
  M_{n,n}}\sum_{i,j}(1-x_{i,j})\ln(1-x_{i,j})+x_{i,j}\ln \left( \frac{a_{i,j}}{x_{i,j}}\right),
\EEA
with the convention $\ln \frac{0}{0}=1$ and where $M_{n,n}$ is the set
of $n\times n$ doubly stochastic matrices.
\end{theorem}
Clearly applying Theorem \ref{th:gurv} to $\tilde{A}$ with
$x_{i,j}=a_{i,j}$, we get Theorem \ref{th:sch} back, so that both
theorems are equivalent.
In \cite{gur06, gur08}, L. Gurvits provided a new proof of these
theorems using stable polynomials, see also \cite{MR2759363}.
Our main new result is a generalization of Theorem
\ref{th:gurv} to subpermanent sums, see Theorem \ref{th:maingen} below.
Our proof is very different from those derived in \cite{sch98, gur06,
  gur08} and borrows ideas from the recent work of Csikv\'ari
\cite{csikvari2014lower} for matchings in regular bipartite graphs.
Permanents and subpermanent sums can be interpreted
as weighted sums of matchings in complete bipartite graphs and the
extension of the approach in \cite{csikvari2014lower} to this general
framework is the main technical contribution of our work. 
Interestingly, the obstacles to overcome are computational in nature
and we present a very algorithmic solution.
As a byproduct, we also obtain new results on the number of matchings in
random lifts, see Theorem \ref{th:rl} below.
We compute the limit of the matching generating function (i.e. the
partition function of the monomer-dimer model) for a sequence of
random $n$-lifts of a graph $G$ as an explicit function of the
original graph $G$. This result has also some implications for the
spectral measure and the matching measure of random lifts, see Theorems
\ref{th:match} and \ref{th:spec}.

As a consequence of Theorem \ref{th:sch}, Schrijver shows in \cite{sch98} that any $d$-regular
bipartite graph with $2n$ vertices has at least
\BEA
\label{eq:matchk}\left(\frac{(d-1)^{d-1}}{d^{d-2}} \right)^n
\EEA
perfect matchings (a perfect matching is a set of disjoint edges
covering all vertices).
For each $d$, the base $(d-1)^{d-1}/d^{d-2}$ in (\ref{eq:matchk}) is
best possible 
\cite{MR2288357}.
Similarly, our Theorem \ref{th:maingen} shows that the right-hand term in (\ref{eq:lnper})
is best possible in the following sense: for any $n\times n$ 
matrix $A$, we show that there exists a sequence of growing
matrices $A_{\ell}$ obtained by taking successive $2$-lifts of $A$
(see Definition \ref{def:liftm}) such that
its permanent grows exponentially with its size at a rate given by the
right-hand term in (\ref{eq:lnper}).
We refer to Theorem \ref{th:maingen} for a precise
statement and similar results for subpermanent sums. 

In \cite{MR2438582}, Friedland, Krop and Markstr\"om conjectured a
possible generalization of (\ref{eq:matchk}) which is known as
Friedland's lower matching conjecture: for $G$ a $d$-regular bipartite
graph with $2n$ vertices, let $m_k(G)$ denote the number of
matchings of size $k$ (see Section \ref{sec:LBG} for a precise
definition), then
\BEA
\label{eq:conjlm}m_{k}(G) \geq {n \choose k}^2\left( \frac{d-p}{d}\right)^{n(d-p)}(dp)^{np},
\EEA
where $p=\frac{k}{n}$.
An asymptotic version of this conjecture was proved using Theorem
\ref{th:sch} in \cite{gur11,6978993}. A slightly stronger statement of
the conjecture
was proved by Csikv\'ari in \cite{csikvari2014lower} and we extend it
to cover 
irregular bipartite graphs, see Theorem \ref{th:main}. 

We state our main results in the next section. 
In Section \ref{sec:idea}, we summarize the main
ideas of the proof and describe related works. We also give some
implications of our work for extremal graph theory and for the
spectral measure and matching measure of random lifts.
Section \ref{sec:proof}
contains the technical proof. We first summarize the statistical
physics results for the monomer dimer model in Section
\ref{sec:statph}. Then, we study local recursions associated to this
model in Section \ref{sec:locrec} The results in this section build mainly on previous work of the author
\cite{lelarge2014loopy}.
In Section \ref{sec:lift}, we show how an idea of Csikv\'ari
\cite{csikvari2014lower} using $2$-lift extends to our framework and
connect it to the framework of local weak convergence in Section
\ref{sec:lwc}. We use probabilistic bounds on the coefficients of
polynomials with only real zeros to finish the proof in Section
\ref{sec:maingen}.
In Section \ref{sec:rlp}, we provide the details needed for random lifts.

\section{Main results}

We present our main results in this section. The results concerning
lower bounds for the number of matchings given in Section
\ref{sec:LBG} are implied by those in Section \ref{sec:LBP} concerning
lower bounds for permanents. The reader only interested in the most
general result concerning lower bounds might jump directly to Section
\ref{sec:LBP} where Theorem \ref{th:maingen} is stated in a
self-contained manner. Note that the proof will use notations
introduced in Section \ref{sec:LBG}.
Section \ref{sec:rl} contains our results on random lifts. They are
independent from the other results but require some notations from Section \ref{sec:LBG}.

\subsection{Lower bounds for number of matchings of a given size}\label{sec:LBG}
We consider a connected multigraph $G=(V,E)$.
We denote by $v(G)$ the cardinality of $V$: $v(G)=|V|$.
 We denote by the same symbol $\d v$ the set of
neighbors of node $v\in V$ and the set of edges incident to $v$.
A matching is encoded by a binary vector, called its incidence vector,
$\bB=(B_e,\: e\in E)\in \{0,1\}^E$ defined by $B_e=1$ if and only if
the edge $e$ belongs to the matching. 
We have for all $v\in V$, $\sum_{e\in \d v}B_e\leq 1$.
The size of the matching is given by $\sum_e B_e$.
We will also use the following notation $e\in \bB$ to mean that $B_e=1$, i.e. that the edge $e$ is in the matching.
For a finite graph $G$, we define the matching number of $G$ as
$\nu(G)=\max\{\sum_e B_e\}$ where the maximum is taken over matchings of
$G$.

The matching polytope $M(G)$ of a graph $G$ is defined as the convex hull of
incidence vectors of matchings in $G$. We define the fractional
matching polytope as
\begin{eqnarray}
\label{def:LG} FM(G) = \left\{ \bx \in \R^E,\: x_e\geq 0,\:\sum_{e\in
    \d v} x_e\leq 1\right\}.
\end{eqnarray}
We also define the fractional matching number $\nu^*(G) = \max_{\bx \in FM(G)}
\sum_e x_e\geq \nu(G)$.
It is well-known that: $M(G)=FM(G)$ if and only if $G$ is bipartite
and in this case, we have $\nu(G)=\nu^*(G)$.

For a given graph $G$, we denote by $m_k(G)$ the number of matchings
of size $k$ in $G$ ($m_0(G)=1$). 
For a parameter $z>0$, we define the matching generating function:
\BEAS
P_G(z) = \sum_{k=0}^{\nu(G)}m_k(G)z^k.
\EEAS
In statistical physics, the function $\ln P_G(z)$ is called the partition function.
We define by $M_k(G)$ the convex hull
of incidence vectors of matchings in $G$ of size $k$ and similarly for $0\leq t\leq \nu^*(G)$:
\BEA
\label{def:fmt}FM_t(G) = \left\{ \bx \in \R^E,\: x_e\geq 0,\:\sum_{e\in
    \d v} x_e\leq 1,\:\sum_{e\in E} x_e=t\right\}.
\EEA
If $G$ is bipartite, we have $M_k(G)=FM_k(G)$.
Theorem \ref{th:main} below deals with bipartite graphs but Theorem \ref{th:rl} deals with general graphs and fractional matching polytopes will be needed there.

For any finite graph $G$, we define the function $S^B_G: FM(G) \to \RR$ by:
 \begin{eqnarray}
\label{def:SB}S^B_G(\bx)&=&\sum_{e\in E} -x_e\ln
x_e +(1-x_e)\ln (1-x_e) - \sum_{v\in V}\left( 1-\sum_{e\in \d v} x_e\right)\ln
\left( 1-\sum_{e\in \d v}x_e\right).
\end{eqnarray}
The function $S^B_G$ is concave on $FM(G)$ by Proposition \ref{prop:conc}.

\begin{definition}
Let $G$ be a graph with no loops. Then $H$ is a $2$-lift of $G$ if $V(H)=V(G) \times
\{0,1\}$ and for every $(u,v)\in E(G)$, exactly one of the following
two pairs are edges of $H$: $((u,0),(v,0))$ and $((u,1),(v,1)) \in
E(H)$ or $((u,0),(v,1))$ and $((u,1),(v,0)) \in E(H)$. If $(u,v)\notin
E(G)$, then none of $((u,0),(v,0))$,$((u,1),(v,1))$, $((u,0),(v,1))$
and $((u,1),(v,0))$ are edges in $H$.
\end{definition}

Note that a bipartite graph $G$ has no loops.
\begin{theorem}\label{th:main}
For any finite bipartite graph $G$, we have for $z>0$,
\BEA
\label{eq:lnP}\ln P_G(z) \geq \max_{\bx \in M(G)} \left\{\left( \sum_e x_e\right)\ln z+S^B_G(\bx)\right\}. 
\EEA
We have for all $k\leq \nu(G)$,
\BEAS
m_k(G) \geq b_{\nu(G),k}(k/\nu(G))\exp \left( \max_{\bx \in
    M_k(G)} S^B_G(\bx)\right),
\EEAS
where 
 $b_{n,k}(p)$ is the probability for a binomial random variable
$Bin(n,p)$ to take the value $k$, i.e. $b_{n,k}(p) = {n \choose k}p^k(1-p)^{n-k}$. 
Moreover, there exists a sequence of bipartite graphs $\{G_n = (V_n,E_n)\}_{n\in
  \N}$ such that $G_0=G$, $G_n$ is a $2$-lift of
$G_{n-1}$ for $n\geq 1$ and for all $z>0$,
\BEAS
\lim_{n\to \infty} \frac{1}{v(G_n)} \ln P_{G_n}(z) = \frac{1}{v(G)}\max_{\bx \in M(G)} \left\{\left( \sum_e x_e\right)\ln z+S^B_G(\bx)\right\}.
\EEAS
\end{theorem}

Consider the particular case where $G$ is a $d$-regular bipartite
graph on $2n$ vertices. In this case, we have $\nu(G)=n$ and we can take $x^*_e=\frac{k}{nd}$
for all $e\in E$ so that $\bx^*\in M_k(G)$ and we have
\BEAS
S^B_G(\bx^*) = n\left(p\ln\left( \frac{d}{p}\right)+(d-p)\ln\left( 1-\frac{p}{d}\right)-2(1-p)\ln(1-p)\right),
\EEAS
with $p=\frac{k}{n}$. We see that we recover the first statement in
Theorem 1.5 of \cite{csikvari2014lower}. In particular, for $k=n$,
i.e. $p=1$, we recover (\ref{eq:matchk}) and for $k<n$, as explained
in \cite{csikvari2014lower}, we slightly
improve upon (\ref{eq:conjlm}). Note that in this particular case, we
have $m_n(G)\leq (d!)^{n/d}$ by a result of Bregman \cite{MR0327788} (see also
\cite{davies2015independent} for upper bounds for $m_k(G)$ with $k\leq
n$).

Taking $z=1$ in (\ref{eq:lnP}), we obtain the following bound on the total number of matchings:
\begin{proposition}\label{prop:totalmatch}
For any bipartite graph $G$, we have:
\BEAS
\sum_{k=0}^{\nu(G)} m_k(G) \geq \exp\left( \max_{\bx \in M(G)}S^B_G(\bx)\right).
\EEAS
\end{proposition}

\subsection{Lower bounds for permanents}\label{sec:LBP}

In this section, we extend previous results to weighted graphs. We
state our results in term of permanents.
Let $A$ be a non-negative $n\times n$ matrix. We denote by $M_n$ the
set of such matrices.
Recall that the permanent of $A\in M_n$ is defined by
\BEAS
\per(A) =\sum_{\sigma\in \Scal_n}\prod_{i=1}^n a_{i, \sigma(i)}.
\EEAS
We define by $M_{n,n}$ the set of $n\times n$ doubly stochastic
matrices:
\BEAS
M_{n,n}=\left\{A,\: 0\leq a_{i,j},\:\sum_i a_{i,j}=\sum_j a_{i,j} =1\right\}\subset M_n.
\EEAS

For $1\leq k\leq n$, let $\per_k(A)$ be the sum of permanents of all
$k\times k$ minors in $A$ and $\per_0(A)=1$. $\per_k(A)$ is called the $k$-th
subpermanent sum of $A$.
We define $M_{n,k}$ the set of $n\times n$ non-negative sub-stochastic
matrices with entrywise $L_1$-norm $k$:
\BEAS
M_{n,k}=\left\{A,\: 0\leq a_{i,j},\:\sum_i a_{i,j}\leq 1,\:\sum_j a_{i,j}
\leq 1,\: \sum_{i,j}a_{i,j}=k\right\}\subset M_n.
\EEAS
We also define the set of substochastic matrices:
\BEAS
M_{n,\leq}=\left\{A,\: 0\leq a_{i,j},\:\sum_i a_{i,j}\leq 1,\:\sum_j a_{i,j}
\leq 1\right\}\subset M_n.
\EEAS
We define the function $S^B:M_{n}\times M_{n,\leq}\to \RR\cup\{-\infty\}$ by
\BEA
\label{def:SBA}S^B(A,\bx) &=& \sum_{i,j}
x_{i,j}\ln\frac{a_{i,j}}{x_{i,j}}+(1-x_{i,j})\ln(1-x_{i,j})\\
&& -\sum_i\left( 1-\sum_jx_{i,j}\right)\ln\left(
  1-\sum_jx_{i,j}\right)-\sum_j\left( 1-\sum_ix_{i,j}\right)\ln\left(
  1-\sum_ix_{i,j}\right),
\EEA
with the convention $\ln \frac{0}{0}=1$. First note that if $A$ is the incidence matrix of a bipartite graph $G$, and $\bx$ is such that there exists $a_{i,j}=0$ and $x_{i,j}>0$, then $S^B(A,\bx)=-\infty$. Moreover if $\bx$ has only non-negative components corresponding to edges of the graph $G$, then we have $S^B(A,\bx)=S^B_G(\bx)$ as defined in (\ref{def:SB}) with a slight abuse of notation: the zero components (on no-edges of $G$) of $\bx$ as argument of $S^B(A,\bx)$ are removed in the argument of $S^B_G(\bx)$.
Note that $\bx\mapsto S^B(A,\bx)$ is concave on $M_{n,\leq}$ (since
$S^B_G$ is concave on $FM(G)$ by Proposition \ref{prop:conc}).

\begin{definition}\label{def:liftm}
Let $A$ be a non-negative $n\times n$ matrix. Then $B$ is a $2$-lift
of $A$ if $B$ is a $2n\times 2n$ non-negative matrix such that for all
$i,j\in \{1,\dots n\}$, either $b_{i,j}=b_{i+n,j+n}=a_{i,j}$ and
$b_{i,j+n}=b_{i+n,j}=0$ or $b_{i,j+n}=b_{i+n,j}=a_{i,j}$ and
$b_{i,j}=b_{i+n,j+n}=0$.
\end{definition}

\begin{theorem}\label{th:maingen}
Let $A$ be a non-negative $n\times n$ matrix. Let $\nu(A)= \max
\{k,\:\per_k(A)>0\}$. For all $k\leq \nu(A)$, we have
\BEA
\label{eq:lowper}\per_k(A) \geq b_{\nu(A),k}(k/\nu(A))\exp\left( \max_{\bx\in M_{n,k}} S^B(A,\bx)\right),
\EEA
where $b_{n,k}(p) = {n \choose k}p^k(1-p)^{n-k}$.
Moreover, there exists a sequence of matrices $\{A_\ell\in M_{2^\ell n}\}_{\ell\in \N}$ such
that $A_0=A$, $A_\ell$ is a $2$-lift of $A_{\ell-1}$ for $\ell\geq 1$
and for all $z>0$,
\BEAS
\lim_{\ell\to \infty}\frac{1}{2^\ell }\ln \left(\sum_{k=0}^{\nu(A_\ell)}
\per_k(A_{\ell})z^k\right) = \max_{\bx\in M_{n,\leq}}\left\{
  \left(\sum_{i,j} x_{i,j}\right)\ln z+S^B(A,\bx)\right\}.
\EEAS
\end{theorem}
If $k=n$ in (\ref{eq:lowper}), we recover Theorem \ref{th:gurv} which is equivalent to
Theorem \ref{th:sch}. Note that if $\nu(A)<n$, then $\per (A)=0$ and
the lower bound in Theorem \ref{th:gurv} is equal to $-\infty$. Indeed if $\per (A)=0$, then if $x\in M_{n,n}$ is a permutation matrix then there exists $i,j$ such that $a_{i,j}=0$ and $x_{i,j}>0$ so that $\ln\left( \frac{a_{i,j}}{x_{i,j}}\right)=-\infty$. The claim then follows from the Birkhoff-von Neumann Theorem which implies that any doubly stochastic matrix can be written as a convex combination of permutation matrices.
Also, results presented in Section \ref{sec:LBG} follow by taking for the matrix $A$, the incidence matrix of the graph $G$.

\subsection{Number of matchings in random lifts}\label{sec:rl}

As in previous section, $G=(V,E)$ is a fixed connected multigraph with
no loops. 
A random $n$-lift of $G$ is a random graph on vertex set $V_1\cup
V_2\cup \dots \cup V_{v(G)}$, where each $V_i$ is a set of $n$
vertices and these sets are pairwise disjoint, obtained by placing a
uniformly chosen random perfect matching between $V_i$ and $V_j$,
independently for each edge $e=ij$ of $G$. We denote the resulting
graph $L_n(G)$.

Our main result shows that the lower bounds derived in Section \ref{sec:LBG} are indeed attained by a sequence of random lifts. More
precisely, we have
\begin{theorem}\label{th:rl}
For any finite graph $G$, we have
\BEAS
\lim_{n\to \infty}\frac{1}{n} \nu(L_n(G)) &=& \nu^*(G)\quad a.s.\\
\forall z> 0,\quad \lim_{n\to \infty}\frac{1}{n}\ln P_{L_n(G)}(z)&=&
\max_{\bx \in FM(G)} \left\{\left( \sum_e x_e\right)\ln z+S^B_G(\bx)\right\}\quad a.s.
\EEAS
where $S^B_G(\bx)$ is defined in (\ref{def:SB}).
We denote by $FPM(G)=FM_{\nu^*(G)}(G)=\{\bx,\:\sum_{e\in \d v }x_e=1\}$
the fractional perfect matching polytope of $G$, then we have
\BEAS
\lim\sup_{n\to \infty} \frac{1}{n} \ln m_{\nu(L_n(G))}(L_n(G)) &\leq&\sup_{\bx\in FPM(G)}S^B_G(\bx).
\EEAS
If, in addition, $G$ is bipartite, then the fractional perfect matching
polytope is simply the perfect matching polytope $PM(G)$ of $G$ and
\BEAS
\lim_{n\to \infty} \frac{1}{n} \ln m_{\nu(L_n(G))}(L_n(G)) &=&\sup_{\bx\in PM(G)}S^B_G(\bx).
\EEAS
\end{theorem}

In \cite{linroz05}, Linial and Rozenman studied the existence of a
perfect matching in $L_n(G)$. Note that if the number of vertices in
$G$ is odd, then in order to have a perfect matching in a $n$-lifts of
$G$, we need to have $n$ even.
For $n$ even, they described a large class of graphs $G$
for which $L_n(G)$ contains a perfect matching asymptotically almost
surely. This class contains all regular graphs and, in
turn, is contained in the class of graphs having a fractional perfect
matching, i.e. graphs $G$ such that $2\nu^*(G)=v(G)$. 
Our result shows that in this case, $L_n(G)$ will contain an almost
perfect matching (possibly missing $o(n)$ vertices) almost surely. If
in addition the graph $G$ is bipartite, the number of such matchings
is exponential in $n$.

In \cite{gjr10}, the number of perfect matchings in $L_n(G)$, denoted
by $\mathrm{pm}(L_n(G))$, is
studied in the limit $n\to \infty$ (along a subsequence ensuring the
existence of a perfect mating), where $G$ is a graph with a
fractional perfect matching. Using the small subgraph conditioning
method, an asymptotic formula for $\EE[\mathrm{pm}(L_n(G))]$ is
computed for any connected regular multigraph $G$ with degree at least
three. Partial results are also given for $\EE[\mathrm{pm}(L_n(G))^2]$
with an explicit formula based on a conjecture (proved only for
$3$-regular graphs).

Note that we always have $\nu(L_n(G))\geq n\nu(G)$ (by lifting a
maximum matching of $G$). In particular, if $G$ has a perfect matching
then $\nu(L_n(G)) = nv(G)/2$. Hence, Theorem \ref{th:rl} implies that
for any $d$-regular bipartite graph $G$, we have
\BEAS
\lim_{n\to \infty} \frac{1}{n} \ln \mathrm{pm} (L_n(G)) &=&
\frac{v(G)}{2}\ln \left(\frac{(d-1)^{d-1}}{d^{d-2}} \right).
\EEAS
This result is consistent with \cite{gjr10}. Indeed by Jensen
inequality, we always have
\BEAS
\EE\left[\ln \mathrm{pm} (L_n(G))\right] \leq \ln\EE\left[ \mathrm{pm} (L_n(G))\right],
\EEAS
and our result shows that in the large $n$ limit, the two quantities
are asymptotically equal.
Note that the fact that $\EE\left[ \mathrm{pm} (L_n(G))^2\right] \sim
\EE\left[ \mathrm{pm} (L_n(G))\right]^2$ to leading exponential order
is not sufficient to prove this asymptotic equality.

A similar result to Theorem \ref{th:rl} was shown for permanent in
\cite{von13,MR3096976}. It is possible to define a random $k$-lift for
a matrix $A$ (by taking the weighted biadjacency matrix of the random
$k$-lift of the weighted graph associated to the biadjacency matrix $A$). 
It is shown in \cite{von13,MR3096976} that $\lim\sup_{k\to\infty} \frac{1}{k}\EE[\ln \per(L_k(A))] = \max_{x\in
  M_{n,n}}\sum_{i,j}(1-x_{i,j})\ln(1-x_{i,j})+x_{i,j}\ln \left( \frac{a_{i,j}}{x_{i,j}}\right)$.
Since \cite{smarandache2015bounding} recently showed that for $A_k$
any $k$-lift of $A$, we have $\per(A_k) \leq \per(A)^k$, they obtained
a new proof of (\ref{eq:lnper}). 

\section{Main ideas of the proof and more related works}\label{sec:idea}

Recall that $\ln P_G(z) =\ln \sum_k m_k(G)z^k$ is called the partition
function.
A crucial observation made by Csikv\'ari \cite{csikvari2014lower} is
that for any $2$-lift $H$ of a bipartite graph $G$, we have
$\frac{1}{v(G)}\ln P_G(z)\geq \frac{1}{v(H)}\ln P_H(z)$, see
Proposition \ref{prop:lift} in the sequel. 
Note that a $2$-lift of a $d$-regular graph is still
$d$-regular. Starting form any $d$-regular bipartite graph $G$,
Csikv\'ari builds a sequence of $2$-lifts with increasing girth. Then
Csikv\'ari uses the framework of local weak convergence in order to
define a limiting graph for the sequence of $2$-lifts.
In this particular case, the limit is the infinite $d$-regular
tree. Using the connection between the local weak convergence and the
matching measure of a graph developed in \cite{abert2014matching},
Csikv\'ari computes a limiting partition function associated to this
infinite $d$-regular tree which is obtained from the Kesten-McKay
measure. This limiting partition function is then a lower bound for
the original partition function $\ln P_G(z)$ and the lower bounds
(\ref{eq:matchk}) and (\ref{eq:conjlm}) are then easily obtained by
properly choosing the parameter $z$ as a function of $k$ the size of
the matchings that we need to count.

Our work extends the recent work of Csikv\'ari done for regular and bi-regular bipartite graphs to irregular bipartite graphs. We are still building a sequence of $2$-lifts with increasing girth and still using the framework of local weak convergence. The limiting object is now the universal cover of the initial graph $G$, i.e. the tree of non-backtracking walks also called the computation tree (see Section \ref{sec:lift} for a precise definition). A direct computation of the matching measure of this (possibly infinite) tree seems tedious. 
Here, we depart significantly from the analysis of Csikv\'ari. Our
approach for the computation of the limiting partition function is
based on an alternative (more algorithmic) characterization first
developed in \cite{bls12} based on local recursions on the universal
cover of $G$. In order to express the computations done on the
universal cover as a function of the original graph $G$, we rely on
results proved by the author in \cite{lelarge2014loopy} where the
local recursions are studied on any finite graph. The solution of
these local recursions on the universal cover (hence a possibly infinite tree)
is in correspondence with the solution of the local recursions on the
initial graph. Since this solution is given by the maximum of a
certain ``entropy-like'' concave function defined by (\ref{def:SB}) on
the fractional matching polytope of the original graph, we obtain an
explicit formula for the limiting partition function and the lower
bound in Theorem \ref{th:main} follows. Our approach is then
generalized to weighted bipartite graphs in order to get our results
for (sub-)permanents (Theorem \ref{th:maingen}) and to random lifts in
order to get Theorem \ref{th:rl}.

As explained in the sequel of this section, all the basic ideas used
in our proofs were present in a form or another in the literature:
using lifts for extremal graph theory was one the main motivation for
their introduction in a series of papers by Amit, Linial, Matousek,
Rozenman and Bilu \cite{MR1883559,linroz05,MR1871947,MR2279667}; the
matching measure and the local recursions already appeared in the
seminal work of Heilmann and Lieb \cite{heilmannlieb}; the function
$S^B_G$ defined in (\ref{def:SB}) is known in statistical physics as
the Bethe entropy \cite{zdeborovamezard}. The main contribution of
this paper is a conceptual message showing how known techniques from
interdisciplinary areas can lead to new applications in theoretical
computer science. In the next subsections, we will try to relate our
results to the existing literature and give credit to the many authors
who inspired our work.

\subsection{Covers, extremal graph theory and message passing algorithms}

The idea to use lifts to build graphs with extremal properties is not new. In \cite{MR2279667}, Bilu and Linial study $2$-lift of $d$-regular graphs in order to construct infinite families of expanders. They showed that the eigenvalues of a $2$-lift are the union of the eigenvalues of the original graph and those of the signing associated to the $2$-lift. They conjectured that every $d$-regular graph has a signing with spectral radius at most $2\sqrt{d-1}$. This conjecture was proved by Marcus, Spielman and Srivastava in \cite{MR3374962} where they construct bipartite Ramnujan graphs of all degree.

We can informally state our Theorem \ref{th:main} as an extremal graph theoretic result: among all bipartite graphs $G$ having universal cover $T$, the universal cover $T$ minimizes the (normalized) partition function $\frac{1}{v(G)}\ln P_G(z)$ for all $z>0$, in particular it minimizes the (normalized) number of matchings $\frac{1}{v(G)} \sum_k m_k(G)$. Of course, $T$ being infinite, the normalized partition function needs to be defined properly and this can be done thanks to the local weak convergence \cite{bls12}. Indeed in the proof of Theorems \ref{th:main} and \ref{th:rl}, we will prove:

\begin{proposition}\label{prop}
Let $G$ be a finite bipartite graph and $T$ be its universal cover. Let $\mathcal{G}$ be the set of finite bipartite graphs with universal cover $T$. We have for all $z>0$,
\begin{eqnarray}
\label{eq:ext}\inf_{G'\in \mathcal{G}}\frac{1}{v(G')}\ln P_{G'}(z) = \frac{1}{v(G)}\max_{\bx \in M(G)} \left\{\left( \sum_e x_e\right)\ln z+S^B_G(\bx)\right\}. 
\end{eqnarray}
Moreover, the sequence $(G_n)$ defined in Theorem \ref{th:main} or the
sequence $L_n(G)$ converge in the local weak sense to $T$ and achieves
the bound (\ref{eq:ext}) in the limit when $n$ tends to infinity.
\end{proposition}
Note that $G_1,G_2\in \mathcal{G}$ if and only if $G_1$ and $G_2$ have
a common finite cover which is a result proved in \cite{MR693362}. Also, the right-hand term in (\ref{eq:ext}) is an invariant of $\mathcal{G}$: this expression will be the same for any graph belonging to $\mathcal{G}$. Indeed our proof will proceed by computing its value thanks to the following message passing algorithm: to each edge $uv\in E$ and time step $t$, we associate two messages $y^t_{u\to v}(z)$ and $y^t_{v\to u}(z)$ obtained by setting $y^0_{u\to v}=y^0_{v\to u}=0$ and for $t\geq 0$,
\BEA
\label{eq:recy}y^{t+1}_{u\to v}(z) = \frac{z}{1+\sum_{w\in \d u\bac v}y^{t}_{w\to u}(z)}.
\EEA
We show that as $t\to\infty$, these iterations will converge to a limit $y_{u\to v}(z)$ and that $x_{uv}(z)=\frac{y_{u\to v}(z) y_{v\to u}(z)}{z+y_{u\to v}(z) y_{v\to u}(z)}$ solves the maximization in (\ref{eq:ext}) (see Propositions \ref{prop:xefin}, \ref{prop:xe}, \ref{prop:max}). Since the recursion (\ref{eq:recy}) is local, the messages obtained after $t$ iterations are the same as the one computed on the computation tree of the graph at depth $t$.
We are able to show that these recursions on infinite trees still have a unique fixed point (see Theorem \ref{th:infint}) so that this fixed point should be the limit (as the number of iterations tend to infinity) obtained by our algorithm runned on the original graph $G$. 
Note that the recursion (\ref{eq:recy}) is well-known and first appeared in the analysis of the monomer dimer problem \cite{heilmannlieb}. It is also used to define a deterministic approximation algorithm for counting matchings in \cite{stoc07}. Indeed, their analysis directly implies that the convergence of our algorithm is exponentially fast in the number of iterations $t$. Note however that the recursion used in \cite{stoc07} corresponds to messages sent on the tree of self-avoiding paths. Instead, we use the tree of non-backtracking paths. The tree of self-avoiding paths is finite and depends on the root whereas our tree is the universal cover of the graph. Also, \cite{stoc07} directly implies the convergence of our message passing algorithm, it does not give any indication about the value of the limit.

At this stage, we should recall that computing the number of matchings falls into the class of $\#P$-complete problems as well as the problem of counting the number of perfect matchings in a given bipartite graph, i.e. computing the permanent of an arbitrary $0-1$ matrix.
By previous discussion, we see that if the graph is locally tree like, then the tree of self-avoiding paths and the universal cover are locally the same, and one can believe that our algorithm will compute a good approximation for counting matchings. This idea was formalized in \cite{zdeborovamezard} and proved rigorously in \cite{bls12} for random graphs. Our Theorem \ref{th:rl}  shows that these results extend to random lifts.
The lower bound in (\ref{eq:lnper})
is called the (logarithm of the) Bethe permanent in the physics literature
\cite{watanabe2010belief,chertkov2013approximating,von13}. 
Similar ideas
using lifts or covers of graphs have appeared in the literature about
message passing algorithms, see \cite{NIPS2012_4649,2013arXiv1309.6859R} and references
therein. We refer to \cite{lelarge2014loopy} for more results
connecting Belief Propagation with our setting.

\subsection{Matching measure and spectral measure of trees}

We now relate our results to the matching measure used by Csikv\'ari
in \cite{csikvari2014lower} and show how our results allow us to compute spectral measure of infinite trees.
The matching polynomial is defined as:
\BEAS
Q_G(z) = \sum_{k=0}^{\nu(G)} (-1)^k m_k(G) z^{n-2k} = z^nP_{G}(-z^{-2}).
\EEAS
We define the matching measure of $G$ denoted by $\rho_G$ as the uniform distribution over the roots of the matching polynomial of $G$:
\BEA
\rho_G = \frac{1}{\nu(G)}\sum_{i=1}^{\nu(G)} \delta_{z_i},
\EEA
where the $z_i$'s are the roots of $Q_G$.
Note that $Q_G(-z) = (-1)^nQ_G(z)$ so that $\rho_G$ is symmetric.

The fundamental theorem for the matching polynomial is the following.
\begin{theorem}\label{th:hl}(Heilmann Lieb \cite{heilmannlieb})
The roots of the matching polynomial $Q_G(z)$ are real and in the interval $[-2\sqrt{D_G-1},2\sqrt{D_G-1}]$, where $D_G$ is the maximal degree in $G$.
\end{theorem}

In particular, the matching measure of $G$ is a probability measure on $\R$.
Of course, the polynomials $P_G(z)$ or $Q_G(z)$ contains the same information as the matching measure $\rho_G$. We can express the quantity of interest in term of $\rho_G$ (see Lemma 8.5 in \cite{heilmannlieb}, \cite{abert2014matching} or \cite{csikvari2014lower}): for $z>0$,
\BEAS
\frac{1}{v(G)}\frac{z P_G'(z)}{P_G(z)} &=& \frac{1}{2}\int \frac{z\lambda^2}{1+z\lambda^2}d\rho_G(\lambda),\\
\frac{\nu(G)}{v(G)} &=& \frac{1}{2}\left(1-\rho_G(\{0\})\right).
\EEAS

As explained above, Csikv\'ari \cite{csikvari2014lower} uses this representation and the fact that for a sequence of $d$-regular graphs converging to a $d$-regular tree, the limiting matching measure is given by the Kesten-MacKay measure, to get an explicit formula for the limiting partition function. Our approach relies on local recursions instead of the connection with the matching measure. Since we are able to solve these recursions, we get the following result for the limiting matching measure.

\begin{theorem}\label{th:match}
Let $G$ be a finite graph and $T(G)$ be its universal cover. For any
sequence of graphs $(G_i)$ with maximal degree $D_G$ and with local weak limit $T(G)$, the matching measure $\rho_{G_i}$ of the graph $G_i$ is weakly convergent to some measure $\mu_{T(G)}$ defined by the formula for $z>0$,
\BEAS
\int \frac{z\lambda^2}{1+z\lambda^2}d\mu_{T(G)}(\lambda) = \frac{2}{v(G)}\left(\sum_e x_e(z)\right),
\EEAS
where the vector $\bx(z)$ is the unique maximizer of $\Phi^B_G(\bx,z)$ in $FM(G)$. Moreover, we have
$\mu_{T(G)}(\{0\}) = 1-2\frac{\nu^*(G)}{v(G)}$.
\end{theorem}
Note that our theorem gives a generating function of the moments of $\mu_{T(G)}$ since for $z>0$ sufficiently small, we have:
\BEAS
\int \frac{z\lambda^2}{1+z\lambda^2}d\mu_{T(G)}(\lambda) = \sum_{i=1}^\infty (-1)^{i+1}{z^i}\int \lambda^{2i}d\mu_{T(G)}(\lambda),
\EEAS
and the series is convergent since the support of all the $\rho_{G_i}$ and hence of $\mu_{T(G)}$ is contained in $[-2\sqrt{D_G-1},2\sqrt{D_G-1}]$.
As shown by Godsil in \cite{god81}, for finite trees, the spectral
measure and the matching measure coincide, this is still true for
infinite trees \cite{res10,bls11,bls12}. In particular, the moments of
the matching measure $\int \lambda^{2i}d\mu_{T(G)}(\lambda)$ can be
interpreted as the average number of closed walks on $T(G)$ where the
average is taken over the starting point of the walk (see Proposition
\ref{prop:lwl} for a precise definition of the random root as the
starting point of the walk).
 
To be more precise, for a finite graph $G$, we denote by $\lambda_1\leq
\dots\leq \lambda_{v(G)}$ the real eigenvalues of its adjacency matrix
and we define the empirical spectral measure of the graph $G$ as the
probability measure on $\RR$:
\BEAS
\mu_G = \frac{1}{v(G)}\sum_{i=1}^{v(G)} \delta_{\lambda_i}.
\EEAS

The following theorem follows from \cite{res10,bls11,bls12} (see also
Chapter 2 in \cite{hdr})
\begin{theorem}\label{th:spec}
Let $G$ be a finite graph and $T(G)$ be its universal cover. For any
sequence of graphs $(G_i)$ with maximal degree $D_G$ and with local
weak limit $T(G)$, the spectral measure $\mu_{G_i}$ of the graph $G_i$
is weakly convergent to the measure $\mu_{T(G)}$ defined in Theorem
\ref{th:match}.
Moreover for all $x\in \RR$, we have $\lim_{i\to \infty}\mu_{G_i}(\{x\})=\mu_{T(G)}(\{x\})$.
\end{theorem}

In particular, we can apply this theorem to  characterize the limiting spectral measure of random
lifts $L_n(G)$ as a function of the original graph $G$. For the atom
at zero, we have:
\BEA
\lim_{n\to \infty} \mu_{L_n(G)}(\{0\}) = 1-2\frac{\nu^*(G)}{v(G)},
\EEA
and Theorem \ref{th:match} allows us to get the moments of the limiting
measure of $\mu_{L_n(G)}$.

\section{Proofs}\label{sec:proof}

\subsection{Statistical physics}\label{sec:statph}

To ease the notation, we will consider a setting with a weighted graph
$G=(V,E)$ with positive weights on edges $\{\theta_e\}_{e\in
  E}$. Taking a bipartite graph $G$ and $\theta_e=1$ for all $e\in E$, we recover the framework of Section \ref{sec:LBG}. To recover the more general framework of Section \ref{sec:LBP}, consider the bipartite graph described by the support of $A$ seen as an incidence matrix and for each $e=(ij)\in E$, define $\theta_e=a_{i,j}$.

We introduce the family of probability distributions on the set of
matchings in $G$ parametrised by a parameter $z>0$:
\BEA
\label{eq:gibbs}\mu^z_G(\bB) = \frac{z^{\sum_e B_e} \prod_{e\in \bB}\theta_e }{P_G(z)},
\EEA
where $P_G(z) =\sum_{\bB}z^{\sum_e B_e}\prod_{e\in \bB}\theta_e\prod_{v\in
  V}\ind\left(\sum_{e\in \d v}B_e\leq 1\right)=\sum_{k=0}^{\nu(G)}
w_k(G) z^k$, with 
\BEAS
w_k(G) = \sum_{\{\bB:\:\sum_eB_e=k\}}\prod_{e\in \bB}\theta_e,
\EEAS
where the sum is over matchings of size $k$.
Note that we have $w_k(G)=\per_k(A)$.
Note also that when $z$ tends to infinity, the measure $\mu_G^z$ converges to the measure:
\BEAS
\mu^\infty_G(\bB) = \frac{\prod_{e\in E}\theta_e}{\per_{\nu(G)}(\theta)},
\EEAS 
which is simply the uniform measure on maximum matchings when $\theta_e=1$ for all edges.
In statistical physics, this model is known as the monomer-dimer model
and its analysis goes back to the work of Heilmann and Lieb
\cite{heilmannlieb}.

We define the following functions:
\begin{eqnarray*}
U^s_G(z) &=& -\sum_{e\in E} \mu_G^z(B_e=1),\\
U^\theta_G(z)&=& \sum_{e\in E}\mu_G^z(B_e=1)\ln\theta_e,\\
S_G(z) &=& -\sum_{\bB} \mu^z_G(\bB)\ln \mu^z_G(\bB).
\end{eqnarray*}
Note that when $\theta_e=1$, we have $U^\theta_G(z)=0$ and $U_G^s$ is called the internal energy while $S_G$ is the canonical entropy.
We now define the partition function $\Phi_G(z)$ by
\begin{eqnarray*}
\Phi_G(z) = -U^s_G(z)\ln z +U^\theta_G(z)  + S_G(z).
\end{eqnarray*}
A more conventional notation in the statistical physics literature corresponds to an inverse temperature $\beta = \ln z$. 
Note that with our definitions, the internal energy $U^s_G(z)$ is negative, equals to minus the
average size of a matching sampled from $\mu_G^z$. This convention is
consistent with standard models in statistical physics where the low
temperature regime minimizes the internal energy, i.e. in our context
maximizes the size of the matching.
A simple computation shows that:
\begin{eqnarray*}
\Phi_G(z) = \ln P_{G}(z) \mbox{ and, } \Phi_G'(z) = \frac{-U_G^s(z)}{z}.
\end{eqnarray*}

\begin{lemma}
The function $U^s_G(z)$ is strictly decreasing and mapping $[0,\infty)$ to
$(-\nu(G), 0]$.
\end{lemma}
\begin{proof}
We have $-U^s_G(z) = \sum_k kw_k(G)z^{k}/P_G(z)$ so that taking the derivative and multiplying by $z$, we get:
\BEAS
-z(U^s_G)'(z) &=& \frac{\sum_k k^2w_k(G)z^k}{P_G(z)}-\left(\frac{\sum_k kw_k(G)z^k}{P_G(z)}\right)^2\\
&=& \sum_k \left(k-\frac{\sum_\ell \ell w_\ell(G)z^\ell}{P_G(z)}\right)^2\frac{w_k(G)z^k}{P_G(z)}>0.
\EEAS 
\end{proof} 

We define $\tau=\tau(G)=2\nu(G)/v(G)$ which is the maximum
fraction of nodes covered by a matching in $G$. Note that $\tau(G)\leq 1$
and $\tau(G)=1$ if and only if the graph $G$ has a perfect matching.
For $t\in [0,\tau)$, we define $z_t(G)\in
[0,\infty)$ as the unique root to $U^s_G(z_t(G))=-tv(G)/2$. Note that
$t\mapsto z_t(G)$ is an increasing function which maps $[0,\tau)$ to $[0,\infty)$.
The function $\cS_G(t)$ is then defined for $t\in [0,\tau)$ by:
\BEA
\label{def:ent}\cS_G(t) = \frac{S_G(z_t(G))+U^\theta_G(z_t(G))}{v(G)},
\EEA
and $\cS_G(t) = -\infty$ for $t>\tau$.

\begin{proposition}
For $t<\tau$, we have $\cS'_G(t) = -\frac{1}{2}\ln z_t(G)$.
The limit $\lim_{t \to \tau}\cS_G(t)$ exists and we define $\cS_G(\tau)=\lim_{t \to \tau}\cS_G(t)=\frac{1}{v(G)}\ln w_{\nu(G)}(G)$.
\end{proposition}
\begin{proof}
We have for $t<\tau$, $\cS_G(t) = \frac{1}{v(G)}\ln P_G(z_t)-t/2\ln z_t$, so that
taking the derivative with respect to $t$, we get:
\BEAS
\cS'_G(t) &=& z'_t\left(\frac{-t}{2z_t}+\frac{P'_G(z_t)}{v(G)P_G(z_t)}\right)
-\frac{\ln z_t}{2}.
\EEAS
Since $U^s_G(z) = -z\frac{P'_G(z)}{P_G(z)}$ and $U^s_G(z_t)=-tv(G)/2$, we get $\cS'_G(t) = -\frac{1}{2}\ln z_t$.
For $t$ large enough, we have $z_t\geq 1$ and the proposition follows.
\end{proof}

The following proposition is proved in \cite{csikvari2014lower} for
unweighted graphs (see Proposition 2.1(g)) but the proof is the same
in the weighted case. We include it her for convenience.
\begin{proposition}\label{prop:ord}
If for some graphs $G_1$ and $G_2$, we have for every $z\geq 0$,
\BEAS
\frac{\Phi_{G_1}(z)}{v(G_1)} \geq \frac{\Phi_{G_2}(z)}{v(G_2)},
\EEAS
then
\BEAS
\cS_{G_1}(t) \geq \cS_{G_2}(t)
\EEAS
for all $0\leq t\leq 1$.
\end{proposition}
\begin{proof}
The assumption ensures that $\frac{\nu(G_1)}{v(G_1)} \geq
\frac{\nu(G_2)}{v(G_2)}$. Moreover if $\frac{\nu(G_1)}{v(G_1)} =
\frac{\nu(G_2)}{v(G_2)}$, then
\BEAS
\frac{\ln w_{\nu(G_1)}(G_1)}{v(G_1)} \geq \frac{\ln w_{\nu(G_2)}(G_2)}{v(G_2)}.
\EEAS
Hence the statement is trivial for
$t\geq 2\nu(G_2)/v(G_2)$.
We consider now $t\in [0, 2\nu(G_2)/v(G_2))$. Note that $\cS_{G_1}(0)=\cS_{G_2}(0)=0$.
The derivative of
$\cS_{G_1}(t)-\cS_{G_2}(t)$ for $t<2\nu(G_2)/v(G_2)$ is 
\BEAS
-\frac{1}{2}\left( \ln z_{t}(G_1)-\ln z_t(G_2)\right)
\EEAS
Assume this derivative is $0$ at $t_0$, then we have
$z_{t_0}(G_1)=z_{t_0}(G_1)=z_0$ and then
\BEAS
\frac{S_{G_1}(z_0)}{v(G_1)} = \frac{\ln P_{G_1}(z_0)}{v(G_1)}-\frac{t_0}{2}\ln
z_0\geq \frac{\ln P_{G_2}(z_0)}{v(G_2)}-\frac{t_0}{2}\ln z_0 = \frac{S_{G_2}(z_0)}{v(G_2)}
\EEAS
Hence the minimums of $\cS_{G_1}(t)-\cS_{G_2}(t)$ on
$[0,2\nu(G_2)/v(G_2))$ are non-negative.
\end{proof}

\subsection{Local recursions on finite graphs and infinite trees}\label{sec:locrec}

Let $G=(V,E)$ be a (possibly infinite) graph with bounded degree and weights on edges $\{\theta_e\}_{e\in E}$.
We introduce the set $\orE$ of directed edges of $G$ comprising two
directed edges $u\to v$ and $v\to u$ for each undirected edge $(uv) \in
E$. For $\ore\in \orE$, we denote by $-\ore$ the edge with opposite
direction. With a slight abuse of notation, we denote by $\d v$ the set of
incident edges to $v\in V$ directed towards $v$. We also denote by
$\d v\bac u$ the set of neighbors of $v$ from which we removed $u$. We also use this notation to denote the set of incident edges to $v$ directed towards $v$ from
which we removed $u\to v$.

Given $G$, we define the map $\Rcal_G:(0,\infty)^{\orE} \to
(0,\infty)^{\orE} $ by $\Rcal_G(\bold{a}) =\bold{b}$ with
\BEAS
b_{u\to v} = \frac{1}{1+\sum_{w\in \d u\bac v}\theta_{wu}a_{w\to u}},
\EEAS
with the convention that the sum over the empty set equals zero.
We also denote by $\Rcal_{u\to v}:(0,\infty)^{\d u\bac v}\to (0,\infty)$ the
local mapping defined by: $b_{u\to v}=\Rcal_{u\to v}(\bold{a})$ (note
that only the coordinates of $\bold{a}$ in $\d u\bac v$ are taken as input
of $\Rcal_{u\to v}$).
Comparisons between vectors are always componentwise.

\begin{proposition}\label{prop:xefin}
Let $G$ be a finite graph. For any $z>0$, the fixed point equation $\by(z) =
z\Rcal_G(\by(z))$ has a unique attractive solution $\by(z)\in (0,+\infty)^{\orE}$.
The function $z\mapsto \by(z)$ is increasing and the
  function $z\mapsto \frac{\by(z)}{z}$ is decreasing for $z>0$.
\end{proposition}

Note that the mapping $z\Rcal_G$ defined in this proposition is simply the
mapping multiplying by $z$ each component of the output of the mapping
$\Rcal_G$ (making the notation consistent).
\begin{proof}
This result is proved for the case $\theta_e=1$ for all edges in
\cite{lelarge2014loopy} (see also \cite{salez12}) and the proof extends to this setting.
\end{proof}
We define for all $v\in V$, the
following function of the vector $(y_{\ore},\:\ore\in \d v)$,
\BEA
\label{def:D1}\Dcal_v(\by) &=& \sum_{\ore\in \d v} \frac{\theta_e y_{\ore}\Rcal_{-\ore}(\by)}{1+\theta_ey_{\ore}\Rcal_{-\ore}(\by)}\\
\label{def:D2}&=& \frac{\sum_{\ore \in \d v} \theta_e y_{\ore}}{1+\sum_{\ore \in \d v} \theta_ey_{\ore}}.
\EEA
Clearly from (\ref{def:D2}), we see that $\Dcal_v$ is an increasing
function of $\by$ and the proposition below follows directly from
the monotonicity of $\by(z)$ proved in Proposition \ref{prop:xefin}:
\begin{proposition}\label{prop:xe}
Let $G=(V,E)$ be a finite graph and $\by(z)$ be the solution to $\by(z) =
z\Rcal_G(\by(z))$. For any $v\in V$, the mapping $z\mapsto
\Dcal_v(\by(z))$ is increasing and $\Dcal_v(\by(z)) = \sum_{e
  \in \d v}x_e(z)$, where
\BEA
\label{def:xe}x_e(z) = \frac{\theta_e y_{\ore}(z)y_{-\ore}(z)}{z+\theta_e y_{\ore}(z)y_{-\ore}(z)}\in (0,1).
\EEA
We denote by $\bx(z) = (x_e(z), \:e\in E)$ the vector defined by
(\ref{def:xe}), then $\bx(z)\in FM(G)$ and we have:
\BEA
\label{eq:limD}\lim_{z\to \infty}\sum_{v\in V} \Dcal_v(\by(z)) = 2\nu^*(G).
\EEA
\end{proposition}
\begin{proof}
The only non-trivial statement in the above proposition is the value
of the limit in (\ref{eq:limD}). In the case $\theta_e=1$, it follows
from Theorem 1 in \cite{lelarge2014loopy} and the proof carries over
to the case $\theta_e>0$.
\end{proof}

For a finite graph $G=(V,E)$ with weights on edges $\{\theta_e\}_{e\in
  E}$, we define for $\bx\in FM(G)$ defined by (\ref{def:LG}) and $z>0$,
\BEAS
U^B_G (\bx) &=& -\sum_{e\in E}x_e,\\
S_G^B (\bx) &=& \sum_{e\in E} x_e\ln \frac{\theta_e}{x_e} +(1-x_e)\ln (1-x_e) - \sum_{v\in V}\left( 1-\sum_{e\in \d v} x_e\right)\ln
\left( 1-\sum_{e\in \d v}x_e\right),\\
\Phi^B_G(\bx,z) &=& - U^B_G(\bx)\ln z +S^B_G(\bx). 
\EEAS
We denote by $\bx(z)$ the vector defined by (\ref{def:xe}) in
Proposition \ref{prop:xe} where $\by(z) = z\Rcal_{G}(\by(z))$.
Note that
\BEA
\label{eq:UBD}U^B_G (\bx(z)) =\frac{-1}{2}\sum_{v\in V} \Dcal_v(\by(z)), 
\EEA
so that by Proposition \ref{prop:xe}, the mapping $z\mapsto U^B_G
(\bx(z))$ is decreasing from $[0,\infty)$ to $(-\nu^*(G),0]$.
Thus, we can define $z_t^B$ as the unique solution in $[0,\infty)$ to
\BEAS
U^B_G(\bx(z^B_t)) = -\frac{tv(G)}{2}, \mbox{ for } t<\tau^*(G)=\frac{2\nu^*(G)}{v(G)}.
\EEAS
Similarly as in (\ref{def:ent}), we define
\BEAS
\cS^B_G(t) = \frac{S^B_G(\bx(z^B_t))}{v(G)} \mbox{ for }
t<\tau^*(G).
\EEAS
Note that we have $\tau^*(G)\geq \tau(G)$ with equality if $G$ is bipartite.

\begin{proposition}\label{prop:max}
Recall that $\bx(z)\in FM(G)$ is defined by (\ref{def:xe}).
We have for any $z>0$,
\begin{eqnarray*}
\sup_{\bx\in FM(G)}\Phi^B_G(\bx;z)=\Phi^B_G(\bx(z);z),
\end{eqnarray*}
and for $t<\tau^*(G)$,
\begin{eqnarray*}
\Sigma^B_G(t) = \frac{1}{v(G)} \max_{\bx\in FM_{tv(G)/2}(G)} S^B_G(\bx),
\end{eqnarray*}
where $FM_t$ is defined in (\ref{def:fmt}) and where the maximum taken over an empty set is equal to $-\infty$.
\end{proposition}
\begin{proof}
The first statement is proved in \cite{lelarge2014loopy} for the
case where $\theta_e=1$ but extends easily to the current framework.
For the second statement, note that for any $\bx\in FM_{tv(G)/2}(G)$ with $t<\tau^*(G)$, we have
\BEAS
\Phi^B_G(\bx,z^B_t) = \frac{tv(G)}{2}\ln z^B_t +S^B_G(\bx)\leq
\Phi^B_G(\bx(z^B_t),z^B_t)= \frac{tv(G)}{2}\ln z^B_t +S^B_G(\bx(z^B_t)).
\EEAS
By definition, we have $\bx(z^B_t)\in M_{tv(G)/2}(G)$, so that $\max_{\bx\in M_{tv(G)/2}(G)} S^B_G(\bx)= S^B_G(\bx(z^B_t))$.
\end{proof}

We now extend Proposition \ref{prop:xefin} to infinite trees:
\begin{theorem}\label{th:infint}
Let $T = (V,E)$ be a (possibly infinite) tree with bounded degree. For each $z>0$, there exists a unique solution in $(0,\infty)^{\orE}$ to the fixed point
equation $\by(z) = z\Rcal_T(\by(z))$, i.e. such that
\BEA
\label{eq:rect}y_{u\to v}(z) = \frac{z}{1+\sum_{w\in \d u\bac v}\theta_{wu}y_{w\to u}(z)}.
\EEA
\end{theorem}
\begin{proof}
First note that any non-negative solution must satisfy $y_{u\to
  v}(z)\leq z$ for all $(uv)\in E$. The compactness of $[0,z]^{\orE}$
(as a countable product of compact spaces) guarantees the existence of
a solution by Schauder fixed point theorem.

To prove the uniqueness, we follow the approach in
\cite{stoc07}. First, we define the change of variable: $h_{u\to v} =
-\ln \frac{y_{u\to v}(z)}{z}$ so that (\ref{eq:rect}) becomes:
\BEA
\label{eq:rech}h_{u\to v} = \ln\left(1+z\sum_{w\in \d u\bac v}\theta_{wu}e^{-h_{w\to u}}\right).
\EEA

We define the function $f:[0,+\infty)^d\mapsto [0,\infty)$ as:
\BEAS
f(\bh)= \ln \left(1+z\sum_{i=1}^k
  \frac{\theta_i}{1+z\sum_{j=1}^{k_i}\theta_j^i e^{-h_j^i}}\right),
\EEAS
where the parameters $k$, $k_i$, $\theta_i$, $\theta_j^i$ and $z$ are
fixed and $d=\sum_{i=1}^kk_i$.

Iterating the recursion (\ref{eq:rech}), we can rewrite it using such
a function $f$ so that uniqueness would be implied if we show that $f$
is contracting.

For any $\bh$ and $\bh'$, we apply the mean value theorem to the
function $f(\alpha\bh+(1-\alpha)\bh')$ so that there exists
$\alpha\in [0,1]$ such that for $\bh_{\alpha}=\alpha
\bh+(1-\alpha)\bh'$,
\BEAS
|f(\bh)-f(\bh')|=|\nabla f(\bh_{\alpha})(\bh-\bh')| \leq \|\nabla f(\bh_{\alpha})\|_{L_1}\|\bh-\bh'\|_{\infty}.
\EEAS

A simple computation shows that:
\BEAS
\|\nabla f(\bh)\|_{L_1} = \frac{z\sum_{i=1}^k\theta_i
  \frac{z\sum_{j=1}^{k_i}\theta_j^ie^{-h^i_j}}{\left(1+z\sum_{j=1}^{k_i}\theta_j^i e^{-h_j^i} \right)^2}}{1+z\sum_{i=1}^k
  \frac{\theta_i}{1+z\sum_{j=1}^{k_i}\theta_j^i e^{-h_j^i}}}.
\EEAS
Let $A_i = \left(1+z\sum_{j=1}^{k_i}\theta_j^i e^{-h_j^i}\right)^{-1}$, then we get
\BEAS
\|\nabla f(\bh)\|_{L_1} =\frac{z\sum_{i=1}^k
  \theta_i(A_i-A_i^2)}{1+z\sum_{i=1}^k \theta_iA_i}=1-\frac{1+z\sum_{i=1}^k \theta_iA_i^2}{1+z\sum_{i=1}^k \theta_iA_i}.
\EEAS
By taking the partial derivatives, we note that this last expression
is maximized when all $A_i$ are equal. Then the solution for the
optimal $A_i$ reduces to a quadratic equation with solution in
$[0,+\infty)$ equals to $A_i =\frac{\sqrt{1+z\Theta}-1}{z\Theta}$,
where $\Theta=\sum_{i=1}^k\theta_i$. Substituting for the maximum
value, we get for any real vector $\bh$,
\BEAS
\|\nabla f(\bh)\|_{L_1}\leq 1-\frac{2}{\sqrt{1+z\Theta}+1}.
\EEAS
\end{proof}

\subsection{$2$-lifts}\label{sec:lift}

If $G$ is a  graph and $v\in V(G)$, the $1$-neighbourhood of $v$ is the
subgraph consisting of all edges incident upon $v$.
A graph homomorphism $\pi:G'\to G$ is a covering map if for each $v'\in
V(G')$, $\pi$ gives a bijection of the edges of the $1$-neighbourhood of
$v'$ with those of $v=\pi(v')$. $G'$ is a cover or a lift of $G$. If edges of $G=(V,E)$ have weights $\theta_e$ then the edges of $G'=(V',E')$ will also have weights with $\theta_{e'}=\theta_{\pi(e')}$.
Note that the definition of $2$-lift for matrices given in Section
\ref{sec:LBP} is consistent with the definition of $2$-lift for graphs
by identifying the matrix $A$ as the weighted incidence matrix of the
bipartite graph.

\begin{proposition}\label{prop:lift}
Let $G$ be a bipartite graph and $H$ be a $2$-lift of $G$. Then
$P_{G}(z)^2\geq P_H(z)$ for $z>0$, $\Sigma_G(t)\geq \Sigma_H(t)$ for
$t\in [0,1]$ and $\nu(H)=2\nu(G)$.
\end{proposition}
\begin{proof}
The proof follows from an argument of Csikv\'ari \cite{csikvari2014lower}.
Note that $G\cup G$ is a particular $2$-lift of $G$ with $P_{G\cup G}(z)= P_G(z)^2$. To prove the first statement of the proposition, we need to show that for any $2$-lift $H$ of $G$, we have: $w_k(G\cup G)\geq w_k(H)$.
Consider the projection of a matching of a $2$-lift of $G$ to $G$. It will consist of disjoint union of cycles of even lengths (since $G$ is bipartite), paths and double-edges when two edges project to the same edge. For such a projection $R=R_1\cup R_2\subset E$ where $R_2$ is the set of double edges, its weight is $\prod_{e\in R_1}\theta_e\prod_{e\in R_2}\theta_e^2$. Now for such a projection, we count the number of possible matchings in $G\cup G$: $n_R(G\cup G)=2^{k(R)}$, where $k(R)$ is the number of connected components of $R_1$. The number of possible matchings in $H$ is $n_R(H)\leq 2^{k(R)}$ since in each component if the inverse image of one edge is fixed then the inverse images of all other edges is also determined. There is no equality as in general not every cycle can be obtained as a projection of a matching of a $2$-lift. For example, if one considers a $8$-cycle as a $2$-lift of a $4$-cycle, then no matching will project on the whole $4$-cycle.

Hence we proved that $w_k(G\cup G)\geq w_k(H)$ so that $P_{G}(z)^2\geq
P_H(z)$ for $z>0$ and the second statement follows from Proposition
\ref{prop:ord}.
For the last statement, since $P_G(z)^2\geq P_H(z)$, we have
$2\nu(G)\geq \nu(H)$ but the opposite inequality is true for any graph
$G$ since a maximum matching in $G$ can be lifted to a matching in $H$
with size twice the size of the original matching.
\end{proof}

Given a graph $G$ with a distinguished vertex $v\in V$, we construct the
(infinite) rooted tree $(T(G),v)$ of non-backtracking walks at $v$ as follows:
its vertices correspond to the finite non-backtracking walks in G
starting in $v$, and we connect two walks if one of them is a one-step
extension of the other. With a slight abuse of notation, we denote by
$v$ the root of the tree of non-backtracking walks started at $v$.
Note that also we constructed $T(G)$ from a particular vertex $v$, this choice is irrelevant.
It is easy to see that $T(G)$ is a cover of $G$, indeed it is the
(unique up to isomorphism) cover of $G$ that is also a cover of every
other cover of $G$. $T(G)$ is called the universal cover of $G$.

Since the local recursions are the same for both
$\Rcal_{T(G)}$ and $\Rcal_G$ and since there is a unique fixed point
for both $z\Rcal_{T(G)}$ and $z\Rcal_G$, the proposition below follows:
\begin{proposition}\label{prop:same}
Let $G$ be a finite graph and $T(G)$ be its universal cover and
associated cover $\pi:T(G)\to G$. By
Propositions \ref{prop:xe} and \ref{prop:xefin}, we can define:
\BEAS
\tilde{\by}(z) =
z\Rcal_{T(G)}(\tilde{\by}(z))\:,\mbox{ and, }\: {\by}(z) =
z\Rcal_G({\by}(z)).
\EEAS
We have $\pi(\tilde{\by}(z))={\by(z)}$,
i.e. $\tilde{y}_{\ore}(z)=y_{\pi(\ore)}(z)$.
\end{proposition}

\subsection{The framework of local weak convergence}\label{sec:lwc}

This section gives a brief account of the framework of local weak
convergence. For more details, we refer to the surveys \cite{aldste, aldlyo}. 

\paragraph{Rooted graphs.} A \textit{rooted graph} $(G,o)$ is a graph $G=(V,E)$ together with a distinguished vertex $o\in V$, called the
\textit{root}. We let $\cGs$ denote the set of all locally finite connected rooted graphs considered up to \textit{rooted isomorphism}, i.e. $(G,o)\equiv(G',o')$ if there exists a bijection $\gamma\colon V\to V'$ that preserves roots ($\gamma(o)=o'$) and 
adjacency ($\{i,j\}\in E\Longleftrightarrow \{\gamma(i),\gamma(j)\}\in E'$). We write $[G,o]_h$ for the (finite) rooted subgraph induced by the vertices lying at graph-distance at
most $h\in\N$ from $o$. The distance
$$\textsc{dist}\left((G,o),(G',o')\right):=\frac{1}{1+r} \ \textrm{ where }\  r=\sup\left\{h\in\N\colon [G,o]_h\equiv [G',o']_h\right\},$$ 
turns $\cGs$ into a complete separable metric space, see \cite{aldlyo}.

With a slight abuse of notation, $(G,o)$ will denote an equivalence
class of rooted graph also called unlabeled rooted graph in graph
theory terminology.
Note that if two rooted graphs are isomorphic, then their rooted trees
of non-backtracking walks are also isomorphic. It thus makes sense to
define $(T(G),o)$ for elements $(G,o)\in \cGs$.

\begin{proposition}\label{prop:lin}
For any graph $G =(V,E)$, there exists a graph sequence
$\{G_n\}_{n\in \N}$ such that $G_0=G$, $G_n$ is a $2$-lift of
$G_{n-1}$ for $n\geq 1$. Hence $G_n$ is a $2^n$-lift of $G$ and we
denote by $\pi_n:G_n\to G$ the corresponding covering.
For any $v\in V$, if $v_n\in \pi_n^{-1}(v)$, we have $(G_n,v_n) \to
(T(G),v)$ in $\cGs$.
\end{proposition}
\begin{proof}
The proof follows from an argument of Nathan Linial \cite{linial_slides}, see also
\cite{csikvari2014lower}.

A random $2$-lift $H$ of a base graph $G$ is the random graph obtained
by choosing between the two pairs of edges $((u,0),(v,0))$ and
$((u,1),(v,1)) \in E(H)$ or $((u,0),(v,1))$ and $((u,1),(v,0)) \in E(H)$
with probability $1/2$ and each choice being made independently.

Let $G$ be a graph with girth $\gamma$ and let $k$ be the number of
cycles in $G$ with size $\gamma$. Let $X$ be the number of
$\gamma$-cycles in $H$ a random $2$-lift of $G$. 
The girth of $H$ must be at least $\gamma$ and a $\gamma$-cycle in $H$
must be a lift of a $\gamma$-cycle in $G$.
A $\gamma$-cycle in $G$ yields: a $2\gamma$-cycle in $H$ with
probability $1/2$; or two $\gamma$-cycles in $H$ with probability
$1/2$. Hence we have $\EE[X]=k$. But $G\cup G$ (the trivial lift) has
$2k$ $\gamma$-cycles. Hence there exists a $2$-lift with strictly less
than $k$ $\gamma$-cycles.
By iterating this step, we see that there exists a sequence $\{G_n\}$
of $2$-lifts such that for any $\gamma$, there exists a $n(\gamma)$
such that for $j\geq n(\gamma)$, the graph $G_j$ has no cycle of
length at most $\gamma$. This implies that for any $v\in V$ and
$v_j\in \pi_j^{-1}(v)$, we have
$\textsc{dist}\left((G_j,v_j),(T(G),v)\right) \leq \frac{2}{\gamma}$
and the proposition follows.
\end{proof}

\paragraph{Local weak limits.} Let $\cP(\cGs)$ denote the set of Borel probability measures on $\cGs$, equipped with the usual topology of weak convergence (see e.g. \cite{billingsley}). Given a finite graph $G=(V,E)$, we construct a random element of $\cGs$ by choosing uniformly at random a vertex
$o\in V$ to be the root, and restricting $G$ to the
connected component of $o$. The resulting law is denoted by $\cU(G)$. If $\{G_n\}_{n\geq 1}$ is a sequence of finite graphs 
such that $\{\cU(G_n)\}_{n\geq 1}$ admits a weak limit $\cL \in\cP(\cGs)$, we call $\cL$ the \textit{local weak limit} of $\{G_n\}_{n\geq 1}$. If $(G,o)$ denotes a random element of $\cGs$ with law $\cL$, we shall use the following slightly abusive notation :
$G_n \lwc (G,o)$ and for $f:\cG_{\star}\to \R$:
\BEAS
\EE_{(G,o)} \left[ f(G,o)\right] = \int_{\cG_{\star}}f(G,o) d\cL(G,o).
\EEAS

As a direct consequence of Proposition \ref{prop:lin}, we get:
\begin{proposition}\label{prop:lwl}
For $G=(V,E)$, let $\{G_n\}_{n\in \N}$ be the sequence of $2$-lifts
defined in Proposition \ref{prop:lin}.
Then $G_n \lwc (T(G),o)$ where $T(G)$ is the universal cover of $G$
with associated cover $\pi:T(G) \to G$ and $o$ is the inverse image of
a uniform vertex $v$ of $G$, $o=\pi^{-1}(v)$.
\end{proposition}

We now state the corresponding well-known result for random lifts:
\begin{proposition}\label{prop:lwrl}
For $G=(V,E)$, let $L_n(G)$ be a random $n$-lift of $G$.
Then $L_n(G) \lwc (T(G),o)$ a.s. where $T(G)$ is the universal cover of $G$
with associated cover $\pi:T(G) \to G$ and $o$ is the inverse image of
a uniform vertex $v$ of $G$, $o=\pi^{-1}(v)$.
\end{proposition}

We are now ready to use the results of the above sections.
The existence of the limits for the partition function, the internal
energy of the monomer-dimer model is known to be continuous for the
local weak convergence (in a much more general setting than here)
\cite{heilmannlieb, bls12, soda12, abert2014matching} but the explicit
expressions given in the right-hand side below are new.

\begin{theorem}\label{th:lwclim}
Let $G$ be a finite graph and $T(G)$ be its universal cover.
Let $(G_n)_{n\geq 1}$ be a sequence 
such that $G_n\lwc (T(G),o)$. We denote by 
$\bx(z)$ the
vector defined by (\ref{def:xe}) in Proposition \ref{prop:xe} where $\by(z) =
z\Rcal_{G}(\by(z))$.
Then we have as $n\to \infty$, for $z>0$,
\BEA
\label{eq:nu}\lim_{n\to \infty} \frac{1}{|V_n|} \nu(G_n) &=& \frac{1}{v(G)}\nu^*(G),\\
\label{eq:PG}\lim_{n\to \infty} \frac{1}{|V_n|} \ln P_{G_n} (z)
&=&\frac{1}{v(G)}\Phi^B_G(\bx(z),z),\\
\label{eq:UG}\lim_{n\to \infty} \frac{1}{|V_n|} U^s_{G_n} (z)
&=&\frac{1}{v(G)} U^B_G(\bx(z)),\\
\label{eq:SG}\lim_{n\to \infty} \frac{1}{|V_n|} \left(S_{G_n} (z)+
  U^\theta_G(z)\right) &=&\frac{1}{v(G)} S^B_{G}(\bx(z)),\\
\label{eq:cS}\lim_{n\to \infty}\cS_{G_n}(t) &=& \cS^B_G(t),\mbox{ for } t<\tau^*(G).
\EEA
\end{theorem}
\begin{proof}
In \cite{heilmannlieb, bls12}, it is shown that the root exposure probability
satisfies (with our notation):
\BEAS
r_{u\to v}(z) = \frac{1}{1+z\sum_{w\in \d u\bac v}\theta_{wu}r_{w\to u}(z)}.
\EEAS
Hence we can use directly results from \cite{bls12} by the simple
change of variable: $y_{u\to v}(z) = zr_{u\to v}(z)$. In particular
Theorem 6 in \cite{bls12} implies that
\BEAS
\lim_{n\to \infty} \frac{1}{|V_n|} U_{G_n} (z) &=&\frac{1}{2}
\EE_{(T(G),o)}\left[1-\frac{1}{1+z\sum_{\ore \in \d o} \theta_er_{\ore}(z)}
\right]\\
&=& \frac{1}{2} \EE_{(T(G),o)}\left[ \frac{\sum_{\ore \in \d o}
    \theta_ey_{\ore}(z)}{1+\sum_{\ore \in \d o} \theta_ey_{\ore}(z)}\right]\\
&=& \frac{1}{2} \EE_{(T(G),o)}\left[ \Dcal_o(\by(z))\right],
\EEAS
and (\ref{eq:UG}) follows from Propositions \ref{prop:same} and
\ref{prop:xe}. (\ref{eq:nu}) follows by taking the limit $z\to \infty$ as shown in Theorem 11 in \cite{bls12} and (\ref{eq:limD}) in Proposition \ref{prop:xe}.

We now prove (\ref{eq:PG}). We start by noting that $\Phi_G'(z) =
\frac{U_G(z)}{z}$ so that the convergence of $\frac{1}{|V_n|} \ln
P_{G_n} (z)$ follows from (\ref{eq:UG}) and Lebesgue dominated
convergence theorem (see Corollary 7 in \cite{bls12}).
We only need to check the validity of the right-hand side expression
in (\ref{eq:PG}).

Note that, we have with $\theta_{\min} = \min_e \theta_e>0$ and $\theta_{\max} = \max_e \theta_e>0$
\BEAS
\frac{1}{|V_n|}\frac{\ln P_{G_n}(z)}{\ln z} \geq  -U^s_{G_n}(z) +
\frac{|E_n|}{|V_n|} \frac{\ln \theta_{\min}}{\ln z},
\EEAS
and since the number of matching is upper bounded by $2^{|E_n|}$, we have
\BEAS
\frac{1}{|V_n|}\frac{\ln P_{G_n}(z)}{\ln z} \leq  -U^s_{G_n}(z)
+\frac{|E_n|}{|V_n|} \frac{\ln \theta_{\max}}{\ln z}+\frac{|E_n|\ln
  2}{|V_n|\ln z}.
\EEAS
Hence, taking first the limit $n\to\infty$ and then the limit $z\to \infty$, we have
\BEAS
\lim_{z\to \infty} \lim_{n\to \infty}\frac{1}{|V_n|}\frac{\ln P_{G_n}(z)}{\ln z} = \frac{\nu^*(G)}{v(G)}.
\EEAS
Since $\frac{1}{v(G)}\Phi^B_G(\bx(z),z) \sim
\frac{\nu(G)}{v(G)} \ln z$ by Proposition \ref{prop:xe} (note that $S_G^B(\bx)$ is bounded), we only need to check that the derivative with respect to $z$ of the right-hand term in
(\ref{eq:PG}) is $\frac{U^B_G(\bx(z))}{z}$.
\begin{lemma}\label{lem}
In the setting of Proposition \ref{prop:xe}, we have
\BEA
\label{eq:xe(z)} \frac{x_e(z)(1-x_e(z))}{z}=  \theta_e\left(1-\sum_{e'\in \d u}x_{e'}(z)\right)\left(1-\sum_{e'\in \d v}x_{e'}(z)\right)
\EEA
\end{lemma}
\begin{proof}
Note that $\sum_{f\in \d v}x_f(z)=\Dcal_v(\by(z))$, so that we have by (\ref{def:D2})
\begin{eqnarray*}
\left(1-\sum_{f\in \d v}x_f(z)\right)&=& \left(1-\frac{\sum_{\ore \in
      \d v} \theta_ey_{\ore}(z)}{1+\sum_{\ore \in \d v} \theta_ey_{\ore}(z)}\right)\\
&=& \left( 1+\sum_{\ore \in \d v} \theta_ey_{\ore}(z)\right)^{-1}
\end{eqnarray*}
We have for $e=(uv)\in E$, 
\BEAS
x_e(z) &=& \frac{\theta_{e}y_{u\to v}(z)}{\frac{z}{y_{v\to u}(z)}+\theta_{e}y_{u\to
    v}(z)},
\EEAS
and using the fact that $\by(z)=z\Rcal_G(\by(z))$, we get
\BEAS
x_e(z)&=&\frac{\theta_e y_{u\to v}(z)}{1+\sum_{w\in \d v}\theta_{wv}y_{w\to v}(z)}= \theta_ey_{u\to
  v}(z)\left(1-\sum_{f\in \d v}x_f(z)\right)\\
1-x_e(z) &=& \frac{1+\sum_{w\in \d u \bac v} \theta_{wu}y_{w\to
    u}(z)}{1+\sum_{w\in \d u}\theta_{wu}y_{w\to u}(z)}=\frac{z}{y_{u\to v}(z)}
  \left(1-\sum_{f\in \d u}x_f(z)\right),
\end{eqnarray*}
and the lemma follows.
\end{proof}

Note that for $e=(uv)$, we have
\BEAS
\frac{\partial \Phi^B_G}{\partial x_e} = \ln z+\ln \left( \theta_e\frac{\left(1-\sum_{f\in \d u}x_f\right)\left(1-\sum_{f\in \d v}x_f\right)}{x_e(1-x_e)}\right).
\EEAS
In particular, we have $\frac{\partial \Phi^B_G}{\partial
  x_e}(\bx(z))=0$ by Lemma \ref{lem} and then $\frac{d \Phi^B_G}{dz}(z)
= -U_G^B(\bx(z))/z$ and (\ref{eq:PG}) follows. Moreover (\ref{eq:SG})
follows from (\ref{eq:PG}) and (\ref{eq:UG}).

We now prove (\ref{eq:cS}).
Assume that there exists an infinite sequence of indices $n$ such that
$z_t(G_{n})\geq z^B_t+\epsilon$.
We denote $z_1=z_t^B$ and $z_2=z_t^B+\epsilon$. We have for
those indices:
\BEAS
-\frac{1}{|V_n|} U^s_{G_n}(z_1) \leq -\frac{1}{|V_n|} U^s_{G_n}(z_2) \leq
-\frac{1}{|V_n|} U^s_{G_n}(z_t(G_{n})) = \frac{t}{2}.
\EEAS
Then by the first part of the proof, we have $-\frac{1}{|V_n|}
U^s_{G_n}(z_1) \to -\frac{1}{v(G)}U^B_G(\bx(z_1)) = \frac{t}{2}$ and $-\frac{1}{|V_n|}
U^s_{G_n}(z_2) \to -\frac{1}{v(G)}U^B_G(\bx(z_2)) > \frac{t}{2}$ by the strict
monotonicity of $z\mapsto U^B_G(\bx(z))$. Hence we obtain a contradiction.
We can do a similar argument for indices such that $z_t(G_{n})\leq
z^B_t-\epsilon$, so that we proved that $z_t(G_n)\to z^B_t$.
Then (\ref{eq:cS}) follows from the continuity of the mappings $z\mapsto \by(z)$ and
$\bx\mapsto S^B_G(\bx)$.
\end{proof}

\begin{proposition}\label{prop:conc}
The function $S^B_G(\bx)$ is non-negative and concave on $FM(G)$.
\end{proposition}
\begin{proof}
From Theorem 20 in \cite{von13}, we know that the function  
\begin{eqnarray*}
h(\bx)&=& -\sum_i x_i\ln x_i +\sum_i(1-x_i)\ln (1-x_i)\\
&&-
\left(1-\sum_{i}x_i\right) \ln \left( 1-\sum_{i}x_i\right)+\left(\sum_{i}x_i\right)\ln\left(\sum_{i}x_i\right)
\end{eqnarray*}
is non-negative and concave on $\Delta^k=\{\bx \in \R^k,\: x_i\geq 0,
\sum_{i=1}^k x_i\leq 1\}$.
Hence the function
\begin{eqnarray*}
g(\bx)&=& -\sum_i x_i\ln x_i +\sum_i(1-x_i)\ln (1-x_i)
-
2\left(1-\sum_{i}x_i\right) \ln \left( 1-\sum_{i}x_i\right)
\end{eqnarray*}
is concave and non-negative on $\Delta^k$ since \begin{eqnarray*}
g(\bx)&=&h(\bx)+H\left(\sum_{i}x_i \right),
\end{eqnarray*}
where $H(p)=-p\ln p -(1-p)\ln (1-p)$ is the entropy of a Bernoulli
random variable and is concave in $p$.
The proposition follows by decomposing the sum in $S^B_G(\bx)$ vertex
by vertex.
\end{proof}

\subsection{Proof of Theorem \ref{th:maingen}}\label{sec:maingen}

\begin{corollary}\label{cor:fin}
Let $G$ be a bipartite graph, then for any $z>0$,
\BEAS
\Phi_G(z)=\ln P_G(z) \geq \max_{\bx\in FM(G)} \Phi^B_G(\bx;z)
\EEAS
and for $t<t^*(G)$, we have
\BEAS
\Sigma_{G}(t) \geq \frac{1}{v(G)} \max_{\bx\in FM_{tv(G)/2}(G)}
S^B_G(\bx).
\EEAS
\end{corollary}
\begin{proof}
We consider the sequence of graphs defined in Theorem
\ref{th:lwclim}. By Proposition \ref{prop:lift}, the sequence
$\{\frac{1}{|V_n|}\Phi_{G_n}(z)\}_{n\in \N}$ is non-increasing in $n$ and converges to
$\frac{1}{v(G)} \Phi^B_G(\bx(z),z)$ by Theorem \ref{th:lwclim}.
Hence the first statement follows from Proposition \ref{prop:max}.

The second statement of Proposition \ref{prop:lift} implies that the sequence
$\{\Sigma_{G_n}(t)\}_{n\in \N}$ is non-increasing in $n$ and converges
to $\Sigma^B_G(t)$ by Theorem \ref{th:lwclim} and the last statement
follows from Proposition \ref{prop:max}.
\end{proof}


The final step for the proof of Theorem \ref{th:maingen} is now a
standard application of probabilistic bounds on the coefficients of
polynomials with only real zeros \cite{pitman}.

Let $k< \nu(G)=\nu$, $t=\frac{2k}{v(G)}$ and $z=z_t(G)$ such that $U^s_G(z)=-tv(G)/2=-k$. For
$i\leq \nu$, we define
\BEAS
a_i = \frac{w_i(G)z^i}{P_G(z)}.
\EEAS

By the Heilmann-Lieb theorem \cite{heilmannlieb}, the polynomial $A(x) =
\sum_{i=0}^{\nu}a_ix^i$ has only real zeros, i.e. $(a_0,\dots,
a_{\nu})$ is a P\'olya Frequency (PF) sequence. Note that
$A(1)=1=\sum_i a_i$.
By Proposition 1 in \cite{pitman}, the sequence $(a_0,\dots,
a_{\nu})$ is the distribution of the number $S$ of successes
in $\nu$ independent trials with probability $p_i$ of success on the
$i$-th trial, where the roots of $A(x)$ are given by $-(1-p_i)/p_i$
for $i$ with $p_i>0$. Note that $\EE[S] = \sum_i ia_i=-U^s_G(z)=k$.

We can now use Hoeffding's inequality see Theorem 5 in \cite{hoef}:
let $S$ be a random variable with probability distribution of the
number of successes in $\nu$ independent trials. Assume that
$\EE[S]=\nu p\in [b,c]$. Then
\BEAS
\PP\left( S\in [b,c]\right) \geq \sum_{i=b}^c{\nu \choose i}p^{i}(1-p)^{\nu-i}.
\EEAS

Hence, we have in our setting with $b=c=k$ and $p=\frac{k}{\nu}$:
\BEAS
\nonumber a_k &\geq& {\nu \choose k} p^k(1-p)^{\nu(1-p)}\\
\label{eq:lowerw}w_k(G) &\geq & b_{\nu,k}(p) \exp\left(v(G)\Sigma_G(t)\right)\\
\nonumber&\geq& b_{\nu,k}(k/\nu) \exp\left(\max_{\bx \in
    FM_{tv(G)/2}(G)} S^B_G(\bx)\right),
\EEAS
where the last inequality follows from Corollary \ref{cor:fin}.

The case $k=\nu$ is easy. Take $t=\frac{2\nu(1-\epsilon)}{v(G)}$
with $\epsilon>0$ and
$z=z_t(G)$ so that $U^s_G(z)=-tv(G)/2=-\nu(1-\epsilon)$. We define the sequence
of $a_i$'s as above. We now have $\EE[S] = \nu(1-\epsilon)$.
We then have $\EE[S]=\sum_i ia_i \leq \nu a_{\nu} + (1-a_{\nu})(\nu -1) =
a_\nu+\nu-1$, so that $a_{\nu}\geq 1-\nu\epsilon$ and
\BEAS
w_{\nu(G)}(G)\geq (1-\nu\epsilon)\exp\left(v(G)\Sigma_G(t)\right)\geq (1-\nu\epsilon)\exp\left(\max_{\bx \in
    FM_{tv(G)/2}(G)} S^B(\bx)\right).
\EEAS
Letting $\epsilon \to 0$ concludes the proof.

\subsection{Proof of Theorem \ref{th:rl}}\label{sec:rlp}

We start with a definition: the perfect matching corresponding to the edge $e$ in
$G$ is called the fibre corresponding to $e$, which we denote by
$F_e$.

We denote $\nu_n = \nu(L_n(G))$.
If $G$ is bipartite, we have by Theorem \ref{th:main} for all $k\leq \nu_n$,
\BEAS
\ln m_{k}(L_n(G)) \geq \ln b_{\nu_n,k}\left( \frac{k}{\nu_n}\right) +
\max_{\bx \in M_{k}(L_n(G))}S^B_{L_n(G)}(\bx).
\EEAS

It is easy to see that
\BEAS
\max_{\bx \in M_{k}(L_n(G))}S^B_{L_n(G)}(\bx) \geq n \max_{\bx\in M_{k/n}(G)}S^B_G(\bx),
\EEAS
since to any $\bx\in M_{k/n}$, we can associate $\by\in M_{k}(L_n(G))$
by taking $y_{e'}=x_e$ for all $e'\in F_e$.
Hence, we get
\BEAS
\frac{1}{n} \ln m_{k}(L_n(G)) \geq \max_{\bx\in M_{k/n}(G)}S^B_G(\bx)+\frac{1}{n}\ln b_{\nu_n,k}\left( \frac{k}{\nu_n}\right) 
\EEAS
Taking $k=\nu_n$ and letting $n\to \infty$, we get
\BEAS
\lim\inf_{n\to \infty} \frac{1}{n} \ln m_{\nu_n}(L_n(G)) \geq \max_{\bx\in FM_{\nu^*(G)}(G)}S^B_G(\bx)
\EEAS

For the upper bound, we do not need to assume that $G$ is bipartite as we have for all $z>0$,
\BEAS
\frac{1}{n}\ln m_{\nu_n}(L_n(G)) \leq \frac{1}{n}\ln P_{L_n(G)}(z)
-\frac{\nu_n}{n} \ln z.
\EEAS
Hence letting $n\to \infty$, we get
\BEAS
\lim\sup_{n\to \infty} \frac{1}{n} \ln m_{\nu_n}(L_n(G)) &\leq&
 \Phi^B_G(\bx(z),z)-\nu^*(G)\ln z \\
&=&\sup_{\bx\in FM(G)}\left\{ \ln z \left(\sum_e x_e-\nu^*(G)\right) +S^B_G(\bx)\right\}
\EEAS
Taking now $z\to \infty$ and noting that $\sum_e x_e-\nu^*(G)\leq 0$, we have:
\BEAS
\lim\sup_{n\to \infty} \frac{1}{n} \ln m_{\nu_n}(L_n(G)) & \leq &\sup_{\bx\in PM(G)}S^B_G(\bx).
\EEAS


\end{document}